\documentclass{article}  

\usepackage{amsmath, amssymb, amsthm}

\usepackage{epsfig}   
\usepackage{xspace}
\usepackage{setspace}
\usepackage[usenames]{color}
\usepackage[latin1]{inputenc}
\usepackage[american]{babel}
\usepackage{stmaryrd}

\usepackage[numbers,sort&compress,longnamesfirst]{natbib}             
\shortcites{Priezzhev1996}

\newcommand{\DIR}{\textsc{dir}}

\newcommand{\INF}{\textsc{inf}}
\newcommand{\CON}{\textsc{con}}

\newcommand{\ARR}{\textsc{arr}}
\newcommand{\NEXT}{\textsc{next}}
\newcommand{\MINUS}{\textsc{minus}}

\newcommand{\ai}{A^{\hspace*{-.05\baselineskip}(i)} }
\newcommand{\ax}[1]{A^{\hspace*{-.05\baselineskip}(#1)} }
\newcommand{\ee}{\varepsilon}

\newcommand{\0}{{\ensuremath{ \mathbf{0}}}}

\newcommand{\x}{{\ensuremath{ \mathbf{x}}}}
\newcommand{\y}{{\ensuremath{ \mathbf{y}}}}
\newcommand{\z}{{\ensuremath{ \mathbf{z}}}}
\newcommand{\A}{{\ensuremath{ \mathbf{A}}}}

\newcommand{\ie}{i.\,e.\xspace}

\newtheorem{thm}{Theorem}  
\newtheorem{lem}[thm]{Lemma}

\newcommand{\thmref}[1]{Theorem~\ref{thm:#1}}
\newcommand{\lemref}[1]{Lemma~\ref{lem:#1}}
\newcommand{\lemrefs}[2]{Lemmas~\ref{lem:#1} and~\ref{lem:#2}}

\newcommand{\secref}[1]{Section~\ref{sec:#1}}

\newcommand{\ineq}[1]{inequality~\eqref{eq:#1}}
\newcommand{\eq}[1]{equation~\eqref{eq:#1}}
\newcommand{\eqs}[2]{equations~\eqref{eq:#1} and~\eqref{eq:#2}}

\newcommand{\Eqs}[2]{Equations~\eqref{eq:#1} and~\eqref{eq:#2}}

\def\mod{\operatorname{mod}}

\DeclareSymbolFont{AMSb}{U}{msb}{m}{n}
\newcommand{\N}{{\mathbb{N}}}
\newcommand{\Q}{{\mathbb{Q}}}
\newcommand{\Z}{{\mathbb{Z}}}
\newcommand{\R}{{\mathbb{R}}}

\allowdisplaybreaks[2]   

\begin{document}

\title{\Large Deterministic Random Walks on Regular Trees}
\author{\qquad Joshua Cooper\thanks{LeConte College,
                               Department of Mathematics,
                               University of South Carolina,
                               1523 Greene Street,
                               Columbia, SC 29208, USA}\qquad \\
        \and
        \qquad Benjamin Doerr\thanks{Department 1: Algorithms and Complexity,
                                Max-Planck-Institut f\"ur Informatik,
				Campus E1 4, 66123 Saarbr\"ucken, Germany}\qquad \\
        \and
        \quad Tobias Friedrich\footnotemark[2]\qquad \\
        \and
        \qquad Joel Spencer\thanks{Courant Institute of Mathematical Sciences,
                              New York University,
                              251 Mercer Street,
                              New York, NY 10012-1185, USA}\qquad}
\date{}

\maketitle

\begin{abstract}
    Jim Propp's rotor router model is a deterministic analogue of a random walk on a
    graph. Instead of distributing chips randomly, 
    each vertex serves its neighbors
    in a fixed order. 
    
    Cooper and Spencer (Comb.~Probab.~Comput.\ (2006)) show a remarkable 
    similarity of both models. If an (almost) arbitrary population of chips is 
    placed on the vertices of a grid $\Z^d$ and does a simultaneous walk in the 
    Propp model, then at all times and on each vertex, the number of chips 
    on this vertex deviates from the expected number the random walk would have gotten there
    by at most a constant. This constant is independent of the starting 
    configuration and the order in which each vertex serves its neighbors.

    This result raises the question if all  graphs do have this property. With 
    quite some effort, we are now able to answer this question negatively. For 
    the graph being an infinite $k$-ary tree ($k \ge 3$), we show that for any 
    deviation $D$ there is an initial configuration of chips such that after 
    running the Propp model for a certain time there is a vertex with at least 
    $D$ more chips than expected in the random walk model. However, to achieve a 
    deviation of $D$ it is necessary that at least $\exp(\Omega(D^2))$ vertices 
    contribute by being occupied by a number of chips not divisible by $k$ at a certain time.
\end{abstract}

\section{Introduction}

The rotor-router model is a simple deterministic process first
introduced by~\citet{Priezzhev1996} and later popularized by Jim Propp.
It can be viewed as an attempt to derandomize random walks on graphs.
So far, the ``Propp machine'' has been studied primarily on infinite grids $\Z^d$.
There, each vertex $x\in\Z^d$ is equipped with a ``rotor'' together with a cyclic permutation
(called a ``rotor sequence'') of the $2d$ cardinal directions of $\Z^d$.
While a chip (particle, coin, \ldots) performing a random walk leaves 
a vertex in a random direction,
in the Propp model it always goes in the direction the rotor is pointing.
After a chip is sent, the rotor is rotated according to the fixed rotor sequence.
This rule ensures that chips are distributed highly evenly among the neighbors of a vertex.

The Propp machine has attracted considerable attention recently.  It has been shown
that it closely resembles a random walk in several respects.
The first results were due to \citet{Levine2} (and later \citet{LevineTree})
who compared random walk and Propp machine in an
\emph{aggregating model} called Internal Diffusion-Limited Aggregation (IDLA).

\citet{CooperSpencer} compared both models
in terms of the \emph{single vertex discrepancy}.
Apart from a technicality, they place
arbitrary numbers of chips on the vertices. Then they run  the Propp machine on
this initial configuration for a certain number of rounds. A round consists of
each chip (in arbitrary order) performing one move as directed by the Propp machine.
For the resulting chip arrangement, they compare the number of chips at each vertex
with the expected number of chips that the random walk would have put there,
starting from the same initial configuration and running for the same time.
Cooper and Spencer showed that for all grids $\Z^d$, these differences can be
bounded  by a constant $c_d$ independent of the initial setup (in particular,
the total number of chips) and the run-time.
For the case $d=1$, that is, the graph being the infinite path,
the optimal constant $c_1$ is approximately $2.29$ \citep{1DPropp},
for the two-dimensional grid it is $c_2\approx 7.87$ \citep{2DPropp}.

This raises the question whether the Propp machine on all graphs simulates the 
random walk that well. In this work, we show that the infinite $k$-regular tree 
behaves different. We prove that here arbitrarily large discrepancies can result 
from suitable initial configurations. However, to obtain a discrepancy of $D$, 
at least $\exp(\Omega(D^2))$ vertices have to participate by being occupied by a number of chips not divisible by $k$ at 
some time.

While the work cited above and ours in this paper primarily aims at 
understanding random walks and their deterministic counterparts from a 
foundations perspective, we would like to mention that meanwhile the rotor router 
mechanism also led to improvements in applications. An example is the 
quasirandom analogue of the randomized rumor spreading protocol to broadcast 
information in networks \cite{BC}.


\section{Preliminaries}
\label{sec:preliminaries}
\label{sec:notations}

To bound the single vertex discrepancy between the Propp machine and a random walk
on the $k$-regular tree we first introduce several requisite definitions and notational conventions.

Let $G=(V,E)$ be the infinite $k$-regular tree, also known as the Cayley tree and the Bethe lattice.
We fix an arbitrary node to be its \emph{origin}~$\0$.
$|\x|$ denotes the distance between the origin and vertex $\x$.

We first describe the \emph{Propp machine} in detail. Each vertex $\x$ is equipped with a rotor pointing to one of its neighbors. If $\A$ is the `direction' the rotor is pointing to, then we denote this neighbor by $\x + \A$. This notation is in analogy to the one used in the grid case. Though we keep the additive notation, here the directions would rather be the generators of the nonabelian group
generated by the expression $\langle \{d_i\}_{i=1}^k \, | \, d_i^2 = 1 \rangle$, whose Cayley graph is an infinite $k$-regular tree. Denote by $\DIR := \{d_1, \ldots, d_k\}$ the set of generators (``directions'').

Each vertex also holds a rule describing how the rotor is moved. This rule is 
encoded by a cyclic permutation of the directions called \emph{rotor sequence}. 
Again, if the rotor of vertex $\x$ is pointing in direction $\A$, then we denote 
by $\NEXT(\A) := \NEXT(\A,\x)$ the new direction after one move. Since there 
will be no danger of confusion, we will always omit the explicit mention of the 
vertex. Nevertheless, we should stress that we do not require all vertices to 
have the same rotor sequence.

Chips move according to the rotors. That is, if at a certain time $t$ there is a 
single chip on vertex $\x$ (with rotor pointing to $\A$), then at time $t+1$ 
this chips is on vertex $\x + \A$, the neighbor the rotor pointed to, and the 
rotor is updated to point in direction $\NEXT(\A)$. If there are more than one 
chip on a vertex, they all move in one round, but applying the rotor principle 
one after the other (in an arbitrary order---we do not care about individual 
chips but only about the number of chips present on each vertex).

With these rules, the Propp machine describes a process that is fully 
determined by the initial setting of the chips and the positions of the rotors. 
For all $\x \in V$ and $t \in \N_0$, we denote by
$f(\x,t)$ the number of chips on vertex~$\x$ and
$\ARR(\x,t)$ the direction of the rotor associated with~$\x$ after
$t$ steps of the Propp machine.  Hence $f(\cdot,0)$ and $\ARR(\cdot,0)$ describe 
the initial configuration, which determines all other values of $f$ and $\ARR$.

For completeness, let us brief{}ly review the \emph{random walk} model. Here, we 
do not have rotors. Instead, in each round, each chip independently and 
uniformly at random chooses a neighbor of his current position and moves to 
there. This process is described by the initial setting of the chips together 
with all random decisions made.

However, since we shall be only interested in the expected number of chips on 
the vertices, we may instead regard the following \emph{linear 
machine}~\cite{CooperSpencer}.
Here, in each time step every vertex splits its pile of chips evenly.
When a pile of $\ell$ chips splits evenly on some vertex,
$\ell/k$ chips go to each neighbor.
By the harmonic property of random walks, the (possibly non-integral) number of chips
at vertex~$\x$ at time~$t$ is exactly the expected number of chips
in the random walk model. 

Note that the linear machine again is a deterministic process fully described by 
the initial numbers of chips on each vertex. Given such an initial 
configuration, we denote by $E(\x,t)$ the (fractional) number of chips at time 
$t$ on vertex $\x$. Note that, again, $E(\cdot,0)$ is just this initial 
configuration. Note further that $E(\x,t)=\tfrac{1}{k} \sum_{\A\in\DIR} 
E(\x+\A,t-1)$ for all $t \in \N$ and $\x \in V$.

To compare the Propp machine with the linear machine, we start both with identical settings of the chips. Hence we have $f(\x,0) = E(\x,0)$ for all $\x \in V$.
A \emph{configuration} describes the current state of  machine we regard. A 
configuration of the Propp machine assigns to each vertex $\x\in V$ its current 
(integral) number of chips and the current direction of the rotor.
For the linear machine, a configuration is a simple mapping 
$V \rightarrow \Q_{\geq0}$, describing the (fractional) number of chips on each vertex.

As pointed out in the introduction, there is one limitation without which
neither the results of \cite{CooperSpencer,1DPropp,2DPropp} nor our results hold.  Note that
since $G$ is a bipartite graph, chips that start on even vertices never
mix with those starting on odd vertices.  It looks as if we are playing
two noninteracting games at once.  However, this is not true. Chips in different bipartition classes
may affect each other through the rotors.   We therefore require the initial configuration
to have chips only on \emph{one} class of the bipartition. Without loss of generality, we consider
only even initial configurations, \ie, chip configurations supported on vertices at an even distance from the origin.

To analyze the behavior of the linear machine, we need the following notation.
By $H(x,t)$ we denote the probability that a chip from a vertex
with distance $x$ to the origin arrives at the origin
after~$t$ random steps (``at time~$t$'')
in a simple random walk.
Then,
\begin{equation}
  H(x,t)=k^{-t} n(x,t) \label{eq:prob}
\end{equation}
with $n(x,t)$ counting the number of paths of length $t$ between two vertices
at distance~$x$ on the infinite $k$-regular tree.
It is easy to verify the following properties of $n(x,t)$:
\begin{equation}
\begin{minipage}[c]{200pt}
\vspace{-\baselineskip}
\begin{align}
    n(0,0) =&\, 1,\notag\\
    n(x,0) =&\, 0 \text{\ \ for all $x\geq1$,}\notag\\
    n(0,t) =&\, k n(1,t-1) \text{\ \ for all $t\geq1$,}\notag\\
    n(x,t) =&\, n(x-1,t-1) + (k-1) n(x+1,t-1) \text{\ \ for all $x,t\geq1$.}\notag
\end{align}
\end{minipage}
\label{eq:n}
\end{equation}

Finally, we write $\x \sim t$ to mean that $|\x| \equiv t \pmod{2}$.


\section{Mod-$k$-forcing Theorem}
\label{sec:lower}

For a deterministic process like the Propp machine, it is obvious that the
initial configuration (that is, the location of each chip and the direction of
each rotor), determines all subsequent configurations. The following theorem
shows a partial converse, namely that (roughly speaking) we may prescribe the
number of chips modulo $k$ on all vertices at all times by finding an appropriate initial
configuration. An analogous result for the one-dimensional
Propp machine has been shown in \cite{1DPropp}.

\begin{thm}[Mod-$k$-forcing Theorem]
    For all initial directions of the rotors
    and any $\pi\colon V\times\N_0\to$ $\{0,1,\ldots,(k-1)\}$ with
    $\pi(\x,t)=0$ for all~$\x\not\sim t$, there is an initial even configuration
    $f(\x,0)$ that results in subsequent configurations satisfying $f(\x,t)\equiv\pi(\x,t)\ (\mod k)$
    for all~$\x$ and~$t \geq 0$.
    \label{thm:parityforcing}
\end{thm}

\newcommand{\fine}{fine\xspace}
\begin{proof}
    Fix arbitrary initial directions for the rotors. 
    Let us call $f$, or more precisely, an initial configuration $(f(\x,0))_{\x \in V}$
    \emph{fine} for a subset $S \subseteq V \times \N_0$,
    \begin{itemize}
        \setlength{\itemsep}{0pt}
        \setlength{\parskip}{0pt}        
        \item if it yields $f(\x,t) \equiv \pi(\x,t)\ (\mod k)$ for all $(\x,t) \in S$, and
        \item if it is even, that is, for all $\x \in V$ we have $f(\x,0) =0$ if $\x \not\sim 0$. 
    \end{itemize}
    Consider first the initial configuration defined by  $f^{(0)}(\x,0):=\pi(\x,0)$ for all $\x \in V$.
    Clearly, $f^{(0)}$ is \fine for $V \times \{0\}$. 
    
    Assume now that there is a $T \ge 0$ and an initial configuration $f^{(T)}$ which is \fine for $V \times \{0, 
    \ldots, T\}$. In order to obtain an $f^{(T+1)}$ that is \fine for \mbox{$V \times \{0, \ldots, T+1\}$}, we shall modify $f^{(T)}$ by defining
    $f^{(T+1)}(\x,0):=f^{(T)}(\x,0)+\ee(\x) \,k^{T+1}$ for appropriately chosen
    $\ee(\x)\in\{0,1,\ldots,(k-1)\}$ with $\ee(\x)=0$ if $\x\not\sim0$.

    Observe that a pile of $k^{T+1}$ chips splits evenly~$T+1$ times. 
    Therefore, for all choices of the $\ee(\x)$, 
    $f^{(T+1)}(\x,t) \equiv f^{(T)}(\x,t)\ (\mod k)$ for all $(\x,t)$ with $t\leq T$.    
    This implies that
    $f^{(T+1)}$ is \fine for $V \times \{0, \ldots, T\}$. 
    
    We shall define $\ee$ inductively such that 
    $f^{(T+1)}$ is also \fine for $V \times \{T+1\}$.
    Let $\ee^{\langle0\rangle}(\x):=0$ for all $\x \in V$.
    Assume that for some $\theta \ge0$,
    $\ee^{\langle\theta\rangle}$ is such that the initial configuration
    $f^{\langle\theta\rangle}(\x,0):=f^{(T)}(\x,0) +\ee^{\langle\theta\rangle}(\x)\,k^{T+1}$ is \fine
    for all $(\x,T+1)$ with $|\x| < \theta$.
    Trivially, this is fulfilled for $\theta=0$.
    For each $\x$ with $|\x|=\theta$,
    we choose some $\y(\x)$ with
    $|\y(\x)|=T+1+\theta$ such that the distance 
    between $\x$ and $\y(\x)$ is $T+1$. Note that this implies that the mapping $\x \mapsto \y(\x)$ is injective.
    We define $\ee^{\langle\theta+1\rangle}$ as follows.
    For all $\x$ with $|\x|=\theta$, 
    let $\ee^{\langle\theta+1\rangle}(\y(\x)):=\big(\pi(\x,T+1)-f^{\langle\theta\rangle}(\x,T+1)\big) \mod k$.    
    For all $\z\in V\setminus \{\y(\x)\colon |\x|=\theta\}$,
    we set $\ee^{\langle\theta+1\rangle}(\z):=\ee^{\langle\theta\rangle}(\z)$.    
    We first argue that $f^{\langle\theta+1\rangle}$ is an even initial configuration.
    As $f^{\langle\theta\rangle}$ is an even initial configuration,
    it suffices to show that for all $\x$ with $|\x|=\theta$ and $\y(\x) \not \sim 0$, 
    $\ee^{\langle\theta+1\rangle}(\y(\x)) = 0$.
    However, if $|\y(\x)|$ is  odd, then $|\x|$ and $T+1$ have different parity.
    Hence $\pi(\x,T+1) = 0$ and $f^{\langle\theta\rangle}(\x,T+1)=0$,
    and by construction, $\ee^{\langle\theta+1\rangle}(\y(\x)) = 0$.
    Moreover, for all $\x$ with $|\x|=\theta$,
    \begin{align*}
        f^{\langle\theta+1\rangle}(\x,T+1)
        &= f^{\langle\theta\rangle}(\x,T+1) +\ee^{\langle\theta+1\rangle}(\y(\x))
        \equiv \pi(\x,T+1)\ (\mod k)
        \intertext{and for all $\x$ with $|\x|<\theta$,}
        f^{\langle\theta+1\rangle}(\x,T+1)
        &=
        f^{\langle\theta\rangle}(\x,T+1)
        \equiv \pi(\x,T+1)\ (\mod k).
    \end{align*}
    Therefore, 
    $f^{\langle\theta+1\rangle}$ is \fine for all $(\x,T+1)$ with $|\x| < \theta+1$.     
    
    For all vertices $\x$ and all $\theta\geq\max(0,|\x| - T)$,
    $\ee^{\langle\theta\rangle}(\x)=
    \ee^{\langle\theta+1\rangle}(\x)$.
    Hence $\ee(\x) := \lim_{\theta \to \infty} \ee^{\langle\theta\rangle}(\x)$ and
    $f^{(T+1)}(\x,0):=f^{(T)}(\x,0)+\ee(\x) \,k^{T+1}$
    are well-defined
    for all~$\x\in V$. Also, $f^{(T+1)}$ is \fine for $V \times \{0, \ldots, T+1\}$.
    
    We now note that
    for all vertices $\x \in V$ and
    for all $T\ge |\x|$, we have that $f^{(T)}(\x,0) = f^{(T+1)}(\x,0)$.
    Hence again, $f(\x,0) := \lim_{T \to \infty} f^{(T)}(\x,0)$ is well-defined
    for all~$\x\in V$.
    Therefore, $f$ is \fine for $V \times \N_0$, which finishes the proof.
\end{proof}


\section{The Basic Method}
\label{sec:basic}

In this section, we lay the foundations for our analysis of the maximal possible
single-vertex discrepancy. Closely following the arguments of~\cite{CooperSpencer}, we will see that it is possible to determine the
contribution of a vertex to the discrepancy at another one independent from
all other vertices.

For the moment, in addition to the notations given in \secref{preliminaries},
we also use the following mixed notation.
By $E(\x,t_1,t_2)$ we denote
the (possibly fractional) number of chips at location $x$
after first performing $t_1$ steps with the Propp machine and
then $t_2-t_1$ steps with the linear machine.

We are interested in bounding the discrepancies
$|f(\x,t)-E(\x,t)|$ for all vertices~$\x$ and all times~$t$.
It suffices to consider the vertex $\x=\0$.
From
\begin{eqnarray*}
E(\0,0,t) &=& E(\0,t),\\
E(\0,t,t) &=& f(\0,t),
\end{eqnarray*}
we obtain
\begin{equation*}
    f(\0,t)-E(\0,t) \ =\ \sum_{s=0}^{t-1} \left( E(\0,s+1,t) - E(\0,s,t) \right).
\end{equation*}

Let $|\x|$ denote the distance of a vertex $\x$ to $\0$.
Now $E(\0,s+1,t) - E(\0,s,t) = \sum_{\x \in V} \sum_{\ell = 1}^{f(\x,s)}
     \big(H(|\x+\NEXT^{\ell-1}(\ARR(\x,s))|, t - s - 1) - H(|\x|,t-s)\big)$
motivates the definition of the \emph{influence} of a Propp move (compared to a
random walk move) from vertex $\x$ in direction $A\in\{-1,+1\}$ on the discrepancy of $\0$
($t$ time steps later) by
\[
    \INF(x,A,t) := H(x+A, t-1) - H(x,t).
\]

In order to ultimately reduce all $\ARR$s involved to the initial arrow settings $\ARR(\cdot,0)$, we  define
$s_i(\x) := \min \big\{ u\geq 0 \mid i < \sum_{t=0}^{u} f(\x,t) \big\}$
for all $i\in\N_0$. Hence
at time $s_i(\x)$ the location~$\x$ is
occupied by its $i$-th chip (where, to be consistent with~\cite{1DPropp}, we start counting with the $0$-th chip).

Consider the discrepancy at $\0$ at time~$T$. Then the above yields

\begin{equation}
    f(\0,T) - E(\0,T) = \sum_{\x\in V} \!\sum_{\substack{i\geq 0,\\ s_i(\x) < T}}\! \INF(|\x|,\NEXT^i(\ARR(\x,0)),T-s_i(\x)).
\label{eq:basic1}
\end{equation}

Since the inner sum of \eq{basic1} will occur frequently in the remainder, let us define the \emph{contribution} of a vertex~$\x$ to be
    \begin{equation*}
        \CON(\x) \ := \sum_{\substack{i\geq 0,\\ s_i(\x) < T}} \INF(|\x|,\NEXT^i(\ARR(\x,0)),T-s_i(\x)),
    \end{equation*}
where we both suppress the initial configuration leading to the $s_i(\cdot)$ as well as the run-time~$T$.

We summarize the discussion so far in the following theorem. It shows that it suffices to examine each vertex~$\x$ separately.
\begin{thm}
    \label{thm:main}
    The discrepancy between Propp machine
    and linear machine after~$T$ time steps is the sum of the contributions $\CON(\x)$
    of all vertices~$\x$,
    \ie,
    \begin{equation*}
        f(\0,T) - E(\0,T) \ =\ \sum_{\x\in V} \CON(\x).
    \end{equation*}
\end{thm}


\section{Divergence of the models}
\label{sec:div}

In this section, we analyze a specific initial configuration and show
that the Propp machine may deviate from the linear machine
by an arbitrarily large number of chips.

We choose the initial directions of the rotors to point inwards in the direction of the origin.
For a fixed time~$T$ at which we aim to maximize the discrepancy
$f(\0,T) - E(\0,T)$ we examine a configuration in which 
all vertices $\x$ with $0 < |\x|\leq T/\lambda$ and $\lambda:=\tfrac{k}{k-2}$
are occupied by a number of chips not divisible by $k$ only once.
We assume that at time $T-t_{|\x|}$ with
$t_{|\x|}:=\llceil\lambda |\x| \rrceil$\footnote{With $\llceil\lambda |\x| \rrceil$
we denote the smallest integer $t$ with $t\geq\lambda |\x|$ and $\x\sim T-t$.}
an ``odd chip'' is sent in the direction of the origin.
That is, the number of chips at $\x$ at time $t_{|\x|}$ is
$f(\x,t_{|\x|})\equiv1$ $(\mod k)$ while 
$f(\x,t)\equiv0$ $(\mod k)$ for all times $t\neq t_{|\x|}$.
For the given initial direction of the rotors
such a configuration exists by \thmref{parityforcing}.
We will prove the following theorem.

\begin{thm}
    \label{thm:div}
    For any initial direction of the rotors and any $T>0$,
    there is an even initial configuration such that the
    single vertex discrepancy between the Propp machine and linear machine
    after $T$ time steps is $\Omega(\sqrt{k\,T})$.
\end{thm}

By \thmref{main}, the discrepancy at the origin at time~$T$
of the above described initial configuration is
\begin{align}
	f(\0,T) - E(\0,T)
	    &= \sum_{\substack{\x\in V,\\ |\x|\leq T/\lambda}} \big( H(|\x|-1,t_{|\x|}-1)-H(|\x|,t_{|\x|}) \big)\notag\\
		&= \sum_{x=1}^{\lfloor T/\lambda \rfloor} k (k-1)^{x-1} \left( H(x-1,t_x-1)-H(x,t_x) \right)\notag\\
		&= \sum_{x=1}^{\lfloor T/\lambda \rfloor} k (k-1)^{x-1} \left( \frac{n(x-1,t_x-1)}{k^{(t_x-1)}} - \frac{n(x,t_x)}{k^{t_x}} \right)\notag\\
		&= \sum_{x=1}^{\lfloor T/\lambda \rfloor} \frac{(k-1)^{x-1}}{k^{t_x-1}} \big( k\,n(x-1,t_x-1) - n(x,t_x) \big)\notag\\
		&= \sum_{x=1}^{\lfloor T/\lambda \rfloor} \frac{(k-1)^{x-1}}{k^{t_x-1}} \ i(x,t_x) \label{eq:si}
\end{align}
with
\[
	i(x,t) := k\,n(x-1,t-1) - n(x,t)
\]
for all $x,t\geq1$.  Let us define $i(x,t)=0$ otherwise.
For $x,t\geq2$ we get
\begin{align*}
	i(x,t) &= k\,n(x-1,t-1) - n(x,t) \\
           &= k n(x-2,t-2) + k (k-1) n(x,t-2) \\
           & \qquad - n(x-1,t-1) - (k-1) n(x+1,t-1) \\
           &= i(x-1,t-1) + (k-1) i(x+1,t-1).
\end{align*}

\noindent
It remains to examine the cases $(x,t)\in(\N\times\N)\setminus(\N_{\geq2}\times\N_{\geq2})$.
Right from the definition, we get $i(1,1)=k-1$.
For $x\geq 2$ and $t=1$ we have
\begin{align*}
    i(x,1) &= k\,n(x-1,0) - n(x,1) \\
           &= 0 \\
           &= i(x-1,0) + (k-1) i(x+1,0).
\end{align*}

\noindent
Also for $x=1$ and $t\geq 2$ we have
\begin{align*}
	i(1,t) &= k\,n(0,t-1) - n(1,t) \\
		   &= (k-1)\,\big(n(0,t-1) - n(2,t-1)\big) \\
		   &= (k-1)\,\big(k\,n(1,t-2) - n(2,t-1)\big) \\
		   &= i(0,t-1) + (k-1)\,i(2,t-1).
\end{align*}

Summarizing the above, we see that $i(x,t)$ can be defined recursively as follows.
\begin{align*}
    i(x,0) &= 0 &\text{for all $x\geq0$,}\\
    i(0,t) &= 0 &\text{for all $t\geq0$,}\\
    i(1,1) &= k-1,\\
    i(x,t) &= i(x-1,t-1) + (k-1)\,i(x+1,t-1) &\text{for $(x,t)\in\N_{\geq 1}^2 \setminus\{(1,1)\}$.}
\end{align*}

This recursive view of $i(x,t)$ reveals another interpretation of these quantities.
Apart from a factor $(k-1)^{(\frac{t-x}{2}+1)}$,
$i(x,t)$ counts the number of lattice paths from the origin $(0,0)$ to $(x,t)$ of steps $(+1,+1)$ and
$(-1,+1)$ which do not cross the line $x=0$.  This can be described by the well-known
\emph{Ballot numbers}.  The classical description is as follows.
Suppose A and B are candidates for president.
Let A receive a total number of$a$~votes and B one of $b$~votes. Let $a\geq b$. Now consider the progress of counting the votes, one after the other. Then the probability that throughout the counting B never has more votes than A is $(a-b+1)/(a+1)$~\citep{Ballot}.
This implies that for given positive integers $a,b$ with $a>b$, the number of lattice paths starting
at the origin and consisting of $a$ upsteps $(+1,+1)$ and $b$ downsteps $(+1,-1)$
such that no step ends on the $x$-axis
is $\frac{a-b}{a+b} \binom{a+b}{a}$.
We are interested in the number of lattice paths starting at $(0,x)$ and
consisting of $(t-x)/2$ upsteps $(+1,+1)$ and $(t+x)/2$ downsteps $(+1,-1)$
such that no step ends below the $x$-axis.
Therefore,
\begin{equation}
  \label{eq:i}
	i(x,t) = (k-1)^{(\frac{t-x}{2}+1)} \,\frac{x}{t} \binom{t}{\frac{t+x}{2}}
\end{equation}
for $x,t>0$.  Note that this can be continued to all $(x,t)\in(\R_+)^2$.

Here is a more intuitive interpretation why the Ballot numbers come into play. 
We want the effect of having one chip 
at distance $x$ and time $t:= t_x$ make its first move toward the origin. 
So we can reduce to {\em comparing} two chips, one at distance $x-1$ 
and the other at distance $x+1$ at time $t-1$.  

We consider the two chips on a common random data set.  The data set 
is a string $s$ of length $t-1$ from $\{I,O\}$ with probabilities 
$\frac{1}{k},\frac{k-1}{k}$ for $I$ (in, toward the origin), $O$ 
(out, away from the origin) respectively.  The string $s$ determines 
where the chip goes (in terms of how far away it is from the origin, 
which is what concerns us): with $O$ the distance $x$ goes up by one 
and with $I$ it goes down by one {\em unless} the current distance 
is zero in which case it goes up by one. 

Say a string $s$ has the property $\MINUS$ if for some initial 
segment the $I$'s exceed the $O$'s by precisely $x$.  In a ``fictitious 
walk'', allowing distance $x$ to go negative, the chip starting at 
$x$ would reach $-1$.  But consider just before this happens for 
the first time.  The two chips are at $0,2$ respectively and an $I$ 
is received.  Now the two chips are both at $1$.  Hence whatever comes 
after, they end up at the same point.  Thus the two chips have the 
same probability of $\MINUS$ and ending at zero. 
    The times that the near chip ends at zero but the far chip does 
not are then precisely the strings $s$ which do not have $\MINUS$ 
and for which the $I$'s exceed the $O$'s by precisely $x-1$.  And as 
the far chip is never closer than the near chip, it never happens 
that the far chip ends at zero but the near chip does not. 
    Thus the difference in the probabilities is given by the Ballot 
problem as calculated above.

\Eqs{si}{i} now give
\begin{align}
	f(\0,T) - E(\0,T) 
		&=\ \sum_{x=1}^{\lfloor T/\lambda \rfloor}
		    \frac{(k-1)^{x-1}}{k^{(t_x-1)}} \
		    (k-1)^{(\frac{t_x-x}{2}+1)} \,
		    \frac{x}{t_x}
		    \binom{t_x}{\frac{t_x+x}{2}}\notag\\
		&=\ \sum_{x=1}^{\lfloor T/\lambda \rfloor}
		    \frac{(k-1)^{(\frac{t_x+x}{2})}}{k^{(t_x-1)}} \
		    \frac{x}{t_x}
		    \binom{t_x}{\frac{t_x+x}{2}}\notag\\
		&>\ \sum_{x=1}^{\lfloor T/\lambda \rfloor}
		    \frac{(k-1)^{\frac{\lambda+1}{2} x}}{2\lambda \, k^{\lambda x}} \
		    \binom{\llceil\lambda x \rrceil}{\frac{\llceil\lambda x \rrceil+x}{2}}.\label{eq:sii}
\end{align}

\noindent
It remains to bound the binomial coefficient.
From the Stirling formula we know that
\begin{equation}
    \frac{5}{2}\,\sqrt{n} \left(\frac{n}{e}\right)^n \,
    <\, n! \,
    <\,\sqrt{\frac{15}{2}}\,\sqrt{n} \left(\frac{n}{e}\right)^n
    \label{eq:stirling}
\end{equation}
for $n>0$ and get
\begin{equation}
    \label{eq:binom}
    \frac{n^{n+1/2}}{3\,(n-k)^{n-k+1/2}\,k^{k+1/2}}
    \,<\, \binom{n}{k}
    \,<\,\frac{n^{n+1/2}}{2\,(n-k)^{n-k+1/2}\,k^{k+1/2}}.
\end{equation}
for $n,k>0$. Note that \eq{binom} also holds for the generalized binomial coefficient
which is defined for nonintegral $n,k$
by using the Gamma function~$\Gamma$.
The following lemma shows that \ineq{sii} also holds without the ceiling function
in the binomial coefficient.

\begin{lem}
    $\dbinom{2x}{x+y}$ is monotonic nondecreasing in $x$ for $0\leq y\leq x$.
    \label{lem:bin}
\end{lem}
\begin{proof}
    With $\Psi$ denoting the logarithmic derivative of the gamma function, \ie,
    the Digamma function, we get
    \begin{align*}
        \frac {\partial \dbinom{2x}{x+y}} {\partial\,x} =
            \big( 2\,\Psi(2x+1) - \Psi(x+y+1) - \Psi(x-y+1) \big) \, \binom{2x}{x+y}.
    \end{align*}
    The lemma follows by the fact that above binomial coefficient is
    positive for $x,y\geq 0$ and that $\Psi(x)$ is monotonic nondecreasing
    for $x>0$ (\citep[Thm.~7]{Inequalities}).
\end{proof}

By the definitions of $\lambda$ and $t_x$ we now get with \lemref{bin} and \eq{binom}
\begin{align}
    \binom{\llceil\lambda x\rrceil}{(\llceil\lambda x\rrceil + x)/2} \geq
    \binom{\lambda x}{\tfrac{\lambda+1}{2} x}
    =\,
    \binom{\tfrac{k}{k-2}x}{\tfrac{1}{k-2}x}
    >\,
        \frac{k^{\frac{k}{k-2}x+\frac{1}{2}} \, (k-2)^\frac{1}{2}}
             {3\,(k-1)^{\frac{k-1}{k-2}x+\frac{1}{2}}\,\sqrt{x}}
\end{align}
for $x>0$, $k>2$.

Using this we obtain for all $k>2$
\begin{align*}
	f(\0,T) - E(\0,T) \,
		>\ \sum_{x=1}^{\lfloor T/\lambda \rfloor} \,\frac{(k-2)^\frac{3}{2}}
		                    {6 \, k^{\frac{1}{2}} \,(k-1)^{\frac{1}{2}}\,\sqrt{x}}
		=\,\Omega(\sqrt{k\,T}),
\end{align*}
which proves \thmref{div}.  Using the same arguments as the following two sections one can also prove
that for all even initial configurations the single vertex discrepancy after
$T$ time steps is at most $O(\sqrt{k T})$.


\section{Convergence of the models}
\label{sec:conv}

The previous section showed that for very special configurations the single vertex
discrepancy can be unbounded.  In this section we show that, on the other hand, many
configurations have a bounded discrepancy.  We will prove the following theorem.

\begin{thm}
    \label{thm:conv}
    If $f(\x,t)\equiv0$ $(\mod k)$ for all $\x$ and $t$ such that
    $(1-\ee) \lambda |\x| < T-t < (1+\ee) \lambda |\x|$
    with $\lambda:=\tfrac{k}{k-2}$, then the
    discrepancy between Propp machine and linear machine
    at time~$T$ and vertex $0$ is bounded by a constant depending only on $k$ and $\ee>0$.
\end{thm}

Our aim is to bound \eq{basic1}.  To that end, we further examine
$\INF$.
From its definition and \eqs{n}{i} we know for $x,t>0$
\begin{align*}
    \INF(x,-1,t) &= \frac{i(x,t)}{k^t},\\
    \INF(x,+1,t) &= \frac{k\,n(x+1,t-1)-n(x,t)}{k^t}\\
                 &= \frac{\frac{1}{k-1} n(x,t) - \frac{k}{k-1} n(x+1,t-1)}{k^t}\\
                 &= \frac{-i(x,t)}{(k-1)\,k^t}.
\end{align*}
This also shows
\begin{equation}
    \INF(x,-1,t) + (k-1)\,\INF(x,1,t) = 0.
    \label{eq:INFfack}
\end{equation}
Therefore, the absolute value of the influence $|\INF(x,A,t)|$
of sending one chip \emph{towards} \0 is $(k-1)$ times
larger than sending one chip in the opposite direction.

Note that
\begin{equation*}
    \CON(\x) \ =\ \sum_{\substack{i\geq 0,\\ s_i(\x) < T}} \ax{i} \frac{i(|\x|,t_i)}{k^{t_i}}
\end{equation*}
with
\begin{align*}
    t_i &:= T-s_i(\x) \\
    \ax{i} &:=
        \begin{cases}
            \frac{-1}{k-1} & \text{for $\NEXT^i(\ARR(\x,0))=+1$}\\
            1              & \text{for $\NEXT^i(\ARR(\x,0))=-1$.}
        \end{cases}
\end{align*}

\noindent
To bound this alternating sum, we use the following elementary fact.
\begin{lem}\label{lem:monotone}
  Let $f\colon X \to \R$ be non-negative and monotone nondecreasing with $X\subseteq\R$.
  Let $\ax{0}, \ldots, \ax{n} \in \R$ and
  $t_0, \ldots, t_n \in X$ such that $t_0 \leq \ldots \leq t_n$
  and
  $|\sum_{i=a}^{b} \ai| \leq 1$ for all $0\leq a\leq b\leq n$.
  Then
    \[\bigg| \sum_{i = 0}^n \ai f(t_i)\bigg| \ \le\  \max_{x \in X} f(x).\]
\end{lem}

\noindent
Let $X \subseteq \R$. We call a
mapping $f\colon X \to \R$ \emph{unimodal}, if there is a $t_1 \in X$ such that
$f|_{x \le t_1}$ as well as $f|_{x \ge t_1}$ are monotone.
The following lemma shows that $\INF(x,A,t)$ is unimodal.

\begin{lem}\label{lem:INFunimodal}
    The function $i(x,t)/k^t$ is unimodal in $t$ for all $x$ and $k$, and it is
    maximized over all $t\in\N$ at some $t_{\max}(x)=\lambda x + c_x$ with $|c_x|\leq 15$.
\end{lem}
\begin{proof}
    Using $\binom{n+2}{k+1} = \frac{n^2+3n+2}{(k+1)\,(n-k+1)}\,\binom nk$ we get
    \begin{align*}
        \frac{i(x,t+2)}{k^{t+2}} - \frac{i(x,t)}{k^{t}} =
              \frac{(k-1)^\frac{t-x}{2} \, x \, p_x(t) }
                      {(t+x+2) \, (t-x+2) \, k^{t+2} \, t}
                 \, \binom{t}{\frac{t+x}{2}}
    \end{align*}
    with
    $
      p_x(t) :=
           - (k^3 - 5k^2 +8k -4)t^2
           - (4k^3 - 8k^2+8k-4)t
           + (k^3-k^2) (x^2-4)
    $.
    Hence, the above difference is non-negative if
    \begin{align}
        \label{eq:pxt}
        t\geq
        \big( \sqrt{8k^3 - 4k^2 - 8k + 4 + k^2 \, (k-2)^2 \, x^2}-
        2k^2 + 2k - 2 \big) \, \big/ \,
        {(k-2)^2}
    \end{align} and non-positive otherwise.
    Thus we have unimodality with $i(x,t)/k^t$ taking its maximum
    when $t$ is the smallest integer having the same parity as $x$ and satisfying \eq{pxt}.
    Using $k\geq 3$, the claim follows.
\end{proof}

Armed with \lemrefs{monotone}{INFunimodal}, we can now prove \thmref{conv}.

\begin{proof}[Proof of Thm.~\ref{thm:conv}]
    By \lemrefs{monotone}{INFunimodal} we know
    \begin{equation*}
    \CON(\x) \leq  \frac{i(|\x|,\lfloor(1-\ee) \lambda |\x|\rfloor)}
                        {k^{\lfloor(1-\ee) \lambda |\x|\rfloor}} +
                   \frac{i(|\x|,\lceil(1+\ee) \lambda |\x|\rceil)}
                        {k^{\lceil(1+\ee) \lambda |\x|\rceil}}.
    \end{equation*}
    By \thmref{main} it therefore remains to show that
    \begin{align*}
        f(\0&,T) - E(\0,T) \\
            &\leq  \sum_{x>0} \left( \frac{(k-1)^{x-1}}{k^{\lfloor(1-\ee) \lambda x\rfloor - 1}} \,
                                     i(x,\lfloor(1-\ee) \lambda x\rfloor) \right . \\
            & \qquad \qquad \left . +
                              \frac{(k-1)^{x-1}}{k^{\lceil(1+\ee) \lambda x\rceil - 1}} \,
                                     i(x,\lceil(1+\ee) \lambda x\rceil) \right)
    \end{align*}
    is bounded.  We now show that the second summand is bounded, the first one can be handled
    analogously.  For this, we choose $\ee'$ such that 
    $(1+\ee') \lambda x=\lceil(1+\ee) \lambda x\rceil$.
    By \eq{i},
    \begin{align*}
        &\frac{(k-1)^{x-1}}{k^{(1+\ee') \lambda x - 1}} \, i(x,(1+\ee') \lambda x)\\
            &\ \ = \frac{(k-1)^{((1+\ee') \lambda+1) x/2}\, x}{k^{(1+\ee') \lambda x - 1} ((1+\ee') \lambda x)}
             \binom{(1+\ee') \lambda x}{\left((1+\ee') \lambda+1\right) x/2}
    \end{align*}
    
    \noindent
    With \eq{binom} we get,
    \begin{align*}
        &\frac{(k-1)^{x-1}}{k^{(1+\ee') \lambda x - 1}} \, i(x,(1+\ee') \lambda x)\\
            &\ \ \ \ < \sqrt{\frac { k \, (k-2)^3  }
                           { (1+\ee') (k \ee' + 2) (2\,k-2+k\ee') }}
               \,p(\ee')^{\frac{x}{k-2}}
    \end{align*}
    with
    \begin{align*}
    p(\ee'):=
            &\frac{2k-2+k\ee'}{(k \ee' + 2)(k-1)}
            \left(\frac{k-1}{(k\ee'+2) (2k-2+k\ee')}\right)^{k\ee'/2}\\
            &\left(\frac{(k-1) \, (2+2\ee')^{\ee'+1}}{2k-2+k\ee'}\right)^k
    \end{align*}
    As $0<p(\ee')<1$ for all $\ee'>0$ and $k\geq3$, we have shown that
    $\sum_{x>0} \frac{(k-1)^{x-1}}{k^{(1+\ee') \lambda x - 1}} \, i(x,(1+\ee') \lambda x)$
    is bounded above by a constant depending on $\ee'>0$.  The same arguments demonstrate that
    $\sum_{x>0} \frac{(k-1)^{x-1}}{k^{\lfloor(1-\ee) \lambda x\rfloor - 1}}\, i(x,\lfloor(1-\ee) \lambda x\rfloor)$
    can also be bounded above by a constant depending on $\ee>0$, which finishes the proof.
\end{proof}


\section{Number of Chips}

In the previous sections, we have shown that the discrepancy on a single vertex can be arbitrarily high provided we use a sufficient amount of time. However, a closer look at the proof also reveals that, assuming $k$ to be constant, to obtain a discrepancy of $D$, at least $\exp(\Omega(D^2))$ vertices have to contribute (by holding a number of chips that is not a multiple of $k$ at a suitable time). We now discuss that also $\exp(\Omega(D^2))$ chips are necessary to obtain a discrepancy of $D$. 

\begin{thm}
    \label{thm:num}
    Let $k \ge 3$ be a constant. Then all single vertex discrepancies arising from an even initial configuration
    with $\kappa$ chips are bounded by $O(\sqrt{\log \kappa})$.
\end{thm}

\begin{proof}
Let us examine the discrepancy at the origin after $T$ time steps, starting with an arbitrary even initial configuration which 
uses only $\kappa$ chips.  We  bound the contribution of each  \emph{sphere}
$S_x:=\{ \x\in V \mid |\x|=x \}$, $x \in \N_0$, separately.

\lemref{INFunimodal} shows that $\INF(x,-1,\cdot)$ is unimodal.
By \lemref{monotone}, this gives
\begin{align}
    \CON(S_x) &:= \sum_{\x\in S_x} \CON(\x) \notag\\
               &\leq \sum_{\x\in S_x} \max_t \INF(|\x|,-1,t) \notag\\
                 &= k (k-1)^{x-1} \max_t \frac{i(x,t)}{k^{t}} \label{eq:sph1}
\end{align}

\noindent
By \eq{i} and 
\lemref{INFunimodal}, which showed that $i(x,t)/k^t$ is maximized
at $t=\lambda x+c_x$ with $|c_x|\leq 15$, we obtain
\begin{align*}
    \max_t \frac{i(x,t)}{k^{t}} 
        =&\ \frac{i(x,\lambda x+c_x)}{k^{\lambda x+c}} \\
        =&\  x\,
            k^{-\lambda x-c_x}\,
            (k-1)^{\frac{(\lambda-1)x+c_x}{2}+1}\,
            (\lambda x+c_x)^{-1}
            \binom{\lambda x+c_x}{\frac{(\lambda+1)x+c_x}2 }. 
\end{align*}
Since both $|c_x|$ and $\lambda=k/(k-2)$ are bounded by absolute constants, we  have
\begin{align*}
    &\binom{\lambda x+c_x}{\frac{(\lambda+1)x+c_x}2 } \Big/
    \binom{\lambda x}{\frac{(\lambda+1)x}2 }\\
    &\ \ \ \ = \frac{ (\lambda x+c_x) \cdot \ldots \cdot (\lambda x+1) }
           { \big[(\frac{\lambda+1}2 x+\frac {c_x} 2) \cdot \ldots \cdot (\frac{\lambda+1}2 x+1)\big] \cdot
             \big[(\frac{\lambda-1}2 x+\frac {c_x} 2) \cdot \ldots \cdot (\frac{\lambda-1}2 x+1)\big]
           }
    =O(1).
\end{align*}
Now \eq{binom} and using that $k$ is seen as constant yields
\begin{align}
    \max_t \frac{i(x,t)}{k^{t}}
        =&\ O\left(
            \frac{ (k-1)^{\frac{(\lambda-1)x}{2}} }
                 { k^{\lambda x}} \,
            \binom{\lambda x}{\frac{(\lambda+1)x}2 }\right) \notag\\
        =&\ O\left(
            \frac{ (k-1)^{\frac{1}{k-2}x} }
                 { k^{\frac{k}{k-2} x} }\,
            \frac{k^{\frac{k}{k-2}x+\frac{1}{2}} \, (k-2)^\frac{1}{2}}
                 {(k-1)^{\frac{k-1}{k-2}x+\frac{1}{2}}\,\sqrt{x}}\right)\notag\\
        =&\ O(x^{-1/2} (k-1)^{-x}). \label{eq:sph2}
\end{align}
From \eqs{sph1}{sph2} we conclude
\begin{align}
    \CON(S_x) &= O(x^{-1/2}). \label{eq:sph3}
\end{align}

\noindent
For large $x$, there are not enough chips such that each vertex on $S_x$ contribute fully to the discrepancy. We use this to obtain a second bound  for $\CON(S_x)$.
\thmref{conv} implies that the sum of all contributions $\CON(S_x)$ for all time steps~$t$
with $T-t >\frac12 \lambda x$ or $T-t < \frac32 \lambda x$ is bounded by a constant $C_k$.
Hence it suffices to examine times $t$ with $\frac12 \lambda x \le T-t \le \frac32 \lambda x$.
In every time step, each chip in $S_x$ can contribute only at most $O( x^{-1/2} k^{-x})$
by \eq{sph2}.
As there are $\kappa$ chips and $\lambda x$ possible time slots, the total 
contribution of $S_x$ in the time interval $\frac12 \lambda x < T-t < \frac32 
\lambda x$ is 
\begin{equation}
    O(\kappa  x^{1/2}\, k^{-x}).
    \label{eq:sph4}
\end{equation}

\noindent
With \thmref{main}, the two bounds of \eqs{sph3}{sph4} now yield
\begin{align*}
    f(\0,T) - E(\0,T)
        &\leq \sum_{x>0} \CON(S_x)\\
        &=    \sum_{1\leq x\leq\log_k(\kappa)} \CON(S_x) + \sum_{x>\log_k(\kappa)} \CON(S_x) \\
        &\le    \sum_{1\leq x\leq\log_k(\kappa)} O(x^{-1/2}) + C_k +
              \sum_{x>\log_k(\kappa)} O(\kappa \, x^{1/2}\, k^{-x}) \\
        &=    O(\sqrt{\log \kappa}) +
              O(\sqrt{\log \kappa}),
\end{align*}
still assuming $k$ to be a constant.  
\end{proof}


\section{Conclusion}

In this paper we showed that $k$-ary trees ($k \ge 3$) do not 
admit a constant bound on the single vertex discrepancies. Nevertheless, also on these trees the Propp machine is a very good simulation of the random walk. For simultaneous walks of $\kappa$ chips, the discrepancies are bounded by $O(\sqrt{\log \kappa})$. 

With this work showing that infinite regular trees do not have discrepancies bounded by a constant, but previous work of  
\citet{CooperSpencer} showing this property for higher-dimensional grids, the natural open problem arising from this work is to give more insight to the question of which graphs  display the one or the other behavior, or ideally, to 
give a characterization of those graphs having constant bounds for the single 
vertex discrepancies.

\section{Acknowledgements}

The third author wishes to thank Philippe Flajolet for some discussions on the 
asymptotics of $n(x,t)$.


\begin{thebibliography}{9}
\providecommand{\natexlab}[1]{#1}
\providecommand{\url}[1]{\texttt{#1}}
\expandafter\ifx\csname urlstyle\endcsname\relax
  \providecommand{\doi}[1]{doi: #1}\else
  \providecommand{\doi}{doi: \begingroup \urlstyle{rm}\Url}\fi

\bibitem[Cooper and Spencer(2006)]{CooperSpencer}
J.~Cooper and J.~Spencer.
\newblock Simulating a random walk with constant error.
\newblock \emph{Combinatorics, Probability and Computing}, 15:\penalty0 815--822, 2006.

\bibitem[Cooper et~al.(2007)Cooper, Doerr, Spencer, and Tardos]{1DPropp}
J.~Cooper, B.~Doerr, J.~Spencer, and G.~Tardos.
\newblock Deterministic random walks on the integers.
\newblock \emph{European Journal of Combinatorics}, 28:\penalty0 2072--2090,
  2007.

\bibitem[Doerr and Friedrich(2009)]{2DPropp}
B.~Doerr and T.~Friedrich.
\newblock Deterministic random walks on the two-dimensional grid.
\newblock \emph{Combinatorics, Probability and Computing}, 18:\penalty0 123--144, 2009.

\bibitem[Doerr et~al.(2008)Doerr, Friedrich, and Sauerwald]{BC}
B.~Doerr, T.~Friedrich, and T.~Sauerwald.
\newblock Quasirandom rumor spreading.
\newblock In \emph{Proceedings of the 19th Annual ACM-SIAM Symposium on
  Discrete Algorithms (SODA~'08)}, pp. 773--781, 2008.

\bibitem[Dragomir et~al.(2000)Dragomir, Agarwal, and Barnett]{Inequalities}
S.~S. Dragomir, R.~P. Agarwal, and N.~S. Barnett.
\newblock Inequalities for beta and gamma functions via some classical and new
  integral inequalities.
\newblock \emph{Journal of Inequalities and Applications}, 5:\penalty0
  103--165, 2000.

\bibitem[Hilton and Pedersen(1991)]{Ballot}
P.~Hilton and J.~Pedersen.
\newblock Catalan numbers, their generalization, and their uses.
\newblock \emph{Mathematical Intelligencer}, 13:\penalty0 64--75, 1991.

\bibitem[Landau and Levine(2009)]{LevineTree}
I.~Landau and L.~Levine.
\newblock The rotor-router model on regular trees.
\newblock \emph{Journal of Combinatorial Theory, Series A}, 116:\penalty0 421--433, 2009.

\bibitem[Levine and Peres(2005)]{Levine2}
L.~Levine and Y.~Peres.
\newblock The rotor-router shape is spherical.
\newblock \emph{Mathematical Intelligencer}, 27:\penalty0 9--11, 2005.

\bibitem[Priezzhev et~al.(1996)Priezzhev, Dhar, Dhar, and
  Krishnamurthy]{Priezzhev1996}
V.~B. Priezzhev, D.~Dhar, A.~Dhar, and S.~Krishnamurthy.
\newblock Eulerian walkers as a model of self-organized criticality.
\newblock \emph{Physical Review Letters}, 77:\penalty0 5079--5082, 1996.

\end{thebibliography}
\end{document}